\documentclass[reqno,a4paper,12pt]{amsart}

\usepackage[all,poly]{xy}
\usepackage{amsfonts,stmaryrd}
\usepackage[mathcal]{eucal}
\usepackage{amssymb}
\usepackage{amsmath}
\usepackage{mathrsfs}
\usepackage{enumerate}

\def\G{\operatorname{G}}

\def\K{\operatorname{K}}

\def\map{\longrightarrow}

\def\GL{\operatorname{GL}}

\def\SL{\operatorname{SL}}

\def\Sp{\operatorname{Sp}}

\def\SK{\operatorname{SK}}

\def\Rat{{\mathbb Q}}

\def\Int{{\mathbb Z}}

\def\pamod#1{\,(\operatorname{mod}{\, #1})\,}

\newtheorem{theorem}{Theorem}
\newtheorem{oldtheorem}{Theorem}

\newtheorem{proposition}{Proposition}
\newtheorem{lemma}{Lemma}
\newtheorem{problem}{Problem}
\newtheorem*{corollary}{Corollary}
\newtheorem{Corollary}{Corollary}

\title[inclusions among commutators of
elementary subgroups]{inclusions among commutators of\\
elementary subgroups}
\author{Nikolai Vavilov}
\address{St. Petersburg State University}
\email{nikolai-vavilov@yandex.ru}
\thanks{This publication is supported by Russian Science Foundation grant 17-11-01261.}
\author{Zuhong Zhang}
\address{Department of  Mathematics, Beijing Institute
of Technology, Beijing, China}
\email{zuhong@hotmail.com}
 \date{}

\begin{document}


\begin{abstract}
In the present paper we continue the study of the elementary
commutator subgroups $[E(n,A),E(n,B)]$, where $A$ and $B$
are two-sided ideals of an associative ring $R$, $n\ge 3$.
First, we refine and expand a number of the auxiliary results, 
both classical ones, due to Bass, Stein, Mason, Stothers, Tits, 
Vaserstein, van der Kallen, Stepanov, as also some of the 
intermediate results in our joint works with Hazrat, and our own 
recent papers \cite{NZ2,NZ3}. The gimmick of the present paper
is an explicit triple congruence for elementary commutators 
$[t_{ij}(ab),t_{ji}(c)]$, where $a,b,c$ belong to three ideals
$A,B,C$ of $R$. In particular, it provides a sharper counterpart 
of the three subgroups lemma at the level of ideals. We derive 
some further striking corollaries thereof, such as a complete 
description of generic lattice of commutator subgroups
$[E(n,I^r),E(n,I^s)]$, new inclusions among multiple
elementary commutator subgroups, etc.
\end{abstract}

\maketitle


\section{Introduction}
 	
Let $\GL(n,R)$  be the general linear group of degree $n\ge 3$
over an associative ring $R$ with 1. For an ideal $A\unlhd R$ we
denote by $\GL(n,R,A)$ the principal congruence subgroup of 
level $A$ and by $E(n,R,A)$ the corresponding relative elementary
subgroup. The study of commutator subgroups 
$$ [\GL(n,R,A),\GL(n,R,B)],\quad [\GL(n,R,A),E(n,R,B)],\quad [E(n,R,A),E(n,R,B)] $$
\noindent
and other related birelative groups
has a venerable history. It goes back to the beginnings of
algebraic $K$-theory in the works of Hyman Bass
\cite{Bass_stable, Bass, Bass_Milnor_Serre}, and was then
continued, at the stable level, by Alec Mason, Wilson Stothers
\cite{Mason_Stothers, Mason_commutators_2}, and many others.
\par
The next breakthrough, for rings satisfying commutativity conditions, 
came with the works by Andrei Suslin, Leonid Vaserstein, Zenon Borewicz
and the first author, Anthony Bak, Alexei Stepanov,
and many others \cite{Suslin, Vaserstein_normal,borvav, Bak,
Stepanov_Vavilov_decomposition}. However, 
these papers mostly addressed only the case where one of the
ideals $A$ or $B$ was the ring $R$ itself. These results depended on new powerful localisation methods introduced by Daniel Quillen and Suslin in connection with Serre's problem, and their off-springs, and 
also on remarkable geometric methods, see \cite{Bak_Vavilov, Hazrat_Vavilov_Bak65} for a systematic description of that stage.
\par
The next stage started with our joint papers with Roozbeh Hazrat 
and Alexei Stepanov \cite{Vavilov_Stepanov_standard, Hazrat_Zhang, Vavilov_Stepanov_revisited}, where we addressed the general case, 
first for $\GL(n,R)$ over commutative rings, then over 
quasi-finite rings. Our works
\cite{RNZ1, RNZ2, Hazrat_Zhang_multiple, RNZ5, Stepanov_universal}
address generalisations to other groups, such as Chevalley groups
and Bak's unitary groups. These results are systematically described 
in our surveys and conference papers
\cite{yoga-1, Porto-Cesareo, yoga-2, Hazrat_Vavilov_Zhang}.
\par
One of the pivotal results of our theory, initially established in a
somewhat weaker form by Hazrat and the second author in
\cite{Hazrat_Zhang_multiple}, and then stated in a more precise
form in our joint papers \cite{Hazrat_Vavilov_Zhang_generation, 
Hazrat_Vavilov_Zhang}, is a description of generators of 
$[E(n,R,A),E(n,R,B)]$ as a subgroup. In those papers it was proven
over quasi-finite rings, and involved three types of generators,
the Stein---Tits--Vaserstein generators $t_{ji}(c)t_{ij}(ab)t_{ji}(-c)$
and $t_{ji}(c)t_{ij}(ba)t_{ji}(-c)$, the elementary commutators
$[t_{ij}(a)t_{ji}(b)]$, where in both cases $i\neq j$, 
$a\in A$, $b\in B$, $c\in R$, and also a third type of generators.
\par
In 2018 we observed that the generators of the third type were
redundant, see, \cite{NV18,NZ1}. Then in 2019 we noticed
that even the elementary commutators should be only taken 
for one root, and finally that this generation result holds
over {\it arbitrary\/} associative rings, see \cite{NZ2}, 
Theorem~1 (which we reproduce below as Lemma~\ref{l3}). This 
proof is then repeated in a slightly more transparent form 
in \cite{NZ3}, \S~5, as part of the proof of the elementary 
multiple commutator formula 
\par
In turn, our proof of that result hinges on the following result,
which is a stronger, and more precise version of \cite{NZ2}, 
Lemma 5. In this form, it is proven as \cite{NZ3}, Lemma 11.
Here we denote by $A\circ B=AB+BA$ the symmetrised product 
of two-sided ideals $A$ and $B$. For commutative rings, $A\circ B=AB=BA$ is the usual product of ideals $A$ and $B$. 

\begin{oldtheorem}
Let $R$ be an associative ring with $1$, $n\ge 3$, and let $A,B$ 
be two-sided ideals of $R$. Then for any  $1\le i\neq j\le n$, any
$1\le k\neq l\le n$, and all $a\in A$, $b\in B$, $c\in R$, one has
$$ y_{ij}(ac,b)\equiv  y_{kl}(a,cb) \pamod{E(n,R,A\circ B)}. $$
\end{oldtheorem}

During the talks  ``commutators of relative and unrelative 
elementary subgroups'' presenting our papers \cite{NZ2}--\cite{NZ5}, both at the algebraic group seminar at Chebyshev Lab on October 22
(see http://chebyshev.spbu.ru/en/schedule/?week=1571605200),
and the at the 
Conference ``Homological algebra, ring theory and Hochschild cohomology'' at EIMI on October 29 (see http://www.pdmi.ras.ru/EIMI/2019/CR/index.html) the first author was invariably writing 
$a\in A$, $b\in B$, $c\in C$. During the second talk, Pavel 
Kolesnikov, who was keeping notes, asked, what was that $C$? 
After a little thinking, we decided that $C$ should
be $C\unlhd R$ here, rather than $R$ itself. Of course, modulo
a smaller subgroup $E(n,R,ABC)$ the elementary commutator 
in the left hand side will be congruent to a product of {\it two\/} 
other elementary commutators, rather than to a single such
commutator, as in the absolute case. 

\begin{theorem}\label{t1}
Let $R$ be an associative ring with $1$, $n\ge 3$, and let $A,B,C$ 
be two-sided ideals of $R$. Then for any  three distinct indices $i,j,h$
such that $1\le i,j,h\le n$, and all $a\in A$, $b\in B$, $c\in C$, 
one has
$$ y_{ij}(ab,c) y_{jh}(ca,b) y_{hi}(bc,a)\equiv 1 
\pamod{E(n,R,ABC+BCA+CAB)}. $$
\end{theorem}
Observe, that the ideal defining the congruence module is
precisely 
$$ ABC+BCA+CAB=A\circ BC+ B\circ CA+C\circ AB. $$
\par
This can be easily established in the same style as in our 
recent papers  \cite{NZ2}--\cite{NZ5}. Morally, it is essentially
the same computation by Mennicke \cite{Mennicke} 
in the form given to it by Bass---Milnor---Serre in \cite{Bass_Milnor_Serre}, Theorem 5.4. Subsequently, it was
used in virtually each and every paper on bounded generation. Alternatively, one could use the Hall---Witt
identity. However, everywhere before $C=R$ so that the last 
factor becomes trivial. 
\par 
Theorem 1, and its immediate consequence Theorem 6,
the three ideals lemma 
$$ [E(n,AB),E(n,C)] \le [E(n,BC),E(n,A)]\cdot [E(n,CA),E(n,B)], $$
\par\noindent
are the high points of the present paper, the rest are either 
preparations to their proof, or corollaries of the above  
{\it trirelative\/} congruence. However, a number of 
intermediate results generalise classical results and are decidedly
interesting in themselves. Let us list other principal results of the 
present paper.
\par\smallskip
$\bullet$ Theorem 2: generators of partially relativised elementary subgroup $E(n,B,A)$, where $A,B\unlhd R$, generalising a classical result by Stein---Tits---Vaserstein.
\par\smallskip
$\bullet$ Theorem 3: reduced set of generators of $E(n,R,A)$, in terms
of the unipotent radicals of two opposite parabolic subgroups, generalising
results by van der Kallen and Stepanov.
\par\smallskip
$\bullet$ Theorem 4: the three ideals lemma for partially relativised
subgroups $E(n,A,BC)$.
\par\smallskip
$\bullet$ Theorem 5: a stable version of the standard commutator formula,
$$ [\GL(n-1,R,A),E(n,R,B)]=[E(n,A), E(n,B)], $$
\noindent
for arbitrary associative rings.
\par\smallskip
$\bullet$ Propositions 2 and 3: new inclusions for multiple elementary commutators.
\par\smallskip
$\bullet$ Theorem 7: a complete description of the generic lattice of inclusions
among $[E(n,I^r),E(n,I^s)]$, for powers of one ideal $I\unlhd R$.
\par\smallskip
$\bullet$ Proposition 5: inclusion
$[E(n,A+B),E(n,A\cap B)]\le [E(n,A),E(n,B)]$.
\par\smallskip
$\bullet$ Proposition 6: an explicit example, where $E(n,R,A\cap B)$ is
strictly smaller than $E(n,R,A)\cap E(n,R,B)$. 
\par\smallskip
The paper is organised as follows. In \S~2 and \S~3 we review
some notation and briefly recall the requisite facts on elementary 
subgroups in $\GL(n,R)$. The next six sections are of a technical 
nature, they develop technical tools for the rest of the present 
paper. Namely, in \S~4 we consider partially relativised elementary subgroups $E(n,B,A)$ and prove Theorem 2, which gives their sets 
of generators. In \S~5 we recall some basic facts on intersections 
of parabolic subgroups with congruence subgroups, after which 
in \S~6 we establish Theorem 3, which is a further strengthening 
of results by van der Kallen and Stepanov, generation of $E(n,R,A)$ 
in terms of unipotent radicals of two opposite parabolic subgroups. 
In \S~7 we prove a toy version of our main results, the three
ideals lemma for partially relativised elementary groups
$E(n,C,AB)$, Theorem~4. In \S~8 we establish Theorem 5,
which is a stable version of the standard commutator formula, 
valid for all associative rings. In \S~9 we discuss the important 
special case, behaviour of elementary commutators modulo 
relative elementary subgroups. The core of the present paper 
is \S~10, where we 
prove Theorem 1 and using that derive an inclusion among 
birelative commutators, Theorem 6. This is kind 
of a three ideals lemma, to be used in all subsequent results.
The balance of this paper is dedicated to applications. In \S~11
we consider an application to the only outstanding case in 
our multiple elementary commutator paper
\cite{NZ3}, quadruple commutators in $\GL(3,R)$, and obtain 
some new inclusions among multiple elementary commutator subgroups.
In \S~12 we obtain definitive results for the crucial case of 
the powers of one ideal and prove Theorem 7. These results 
will be instrumental in the sequel of the present paper dedicated 
to the case of Dedekind rings.
In \S~13 we compare the commutator of two elementary
subgroups of levels $A$ and $B$ with the commutator of 
elementary subgroups of levels $A\cap B$ and $A+B$.
In \S~14 we construct a counter-example concerning 
intersections of relative elementary subgroups.
Finally, in \S~15 we make some further related observations,
and state some unsolved problems.
\par
Initially, we planned to include in this paper also explicit
computations over Dede\-kind rings. But then we realised
that the topic is so extensive that it would be more appropriate
to publish those results separately.


\section{Notation}

\subsection{Commutators}
Let $G$ be a group. A subgroup $H\le G$ generated by a
subset $X\subseteq G$ will be denoted by $H=\langle X\rangle$.
For two elements $x,y\in G$ we denote by ${}^xy=xyx^{-1}$
and $y^x=x^{-1}yx$ the left and right conjugates of $y$ by $x$,
respectively. Further, we denote by 
$$ [x,y] = xyx^{-1}y^{-1}={}^xy\cdot y^{-1}=x\cdot{}^y{x^{-1}} $$
\noindent
the left-normed commutator of $x$ and $y$. Our multiple 
commutators are also left-normed. Thus, by default, $[x,y,z]$
denotes $[[x,y],z]$, we will use different notation for other
arrangement of brackets. Throughout the present paper we
repeatedly use the customary commutator identities, such as 
their multiplicativity with respect to the factors:
$$ [x,yz]=[x,y]\cdot{}^y[x,z],\qquad [xy,z]={}^x[y,z]\cdot[x,z], $$
\noindent
and a number of other similar identities, such as 
$${ [x,y]}^{-1}=[y,x],\qquad {}^z[x,y]=[{}^zx,{}^zy],\qquad 
[x^{-1},y]=[y,x]^x,\qquad [x,y^{-1}]=[y,x]^y, $$
\noindent
usually without any specific reference.
Iterating multiplicativity we see that the commutator 
$[x_1\ldots x_m,y]$ is the product of conjugates of the 
commutators $[x_i,y]$, $i=1,\ldots,m$. Obviously, a similar
claim holds also for $[x,y_1\ldots y_m]$.
\par
Further, for two subgroups $F,H\le G$ one denotes by $[F,H]$ 
their mutual commutator subgroup, spanned by all commutators 
$[f,h]$, where $f\in F$, $h\in H$. Clearly, $[F,H]=[H,F]$, and 
if $F,H\unlhd G$ are normal in $G$, then $[F,H]\unlhd G$ is also 
normal. Similarly, for $F,H,K\le G$ three subgroups of $G$ their
triple commutator $[F,H,K]$ is spanned by $[f,h,k]$, where
$f\in F$, $h\in H$ and $k\in K$. We will use the following version
of the three subgroups lemma.

\begin{lemma}\label{l1}
If $F,H,K\le G$ be three subgroups of $G$. Assume
that two of the subgroups $[F,H,K]$, $[H,K,F]$, $[K,F,H]$ 
are normal in $G$. Then the third of them is also normal and
$$ [F,H,K]\le [H,K,F]\cdot [K,F,H]. $$ 
\end{lemma}

Often times, elementary textbooks needlessly assume
that the subgroups $F,H,K$ themselves are normal in $G$. This
depends, of course, on the exact form of the Hall---Witt identity
one is using. In the correct form, the only conjugations occur
outside of the commutators, one such form is
$$ [x,y^{-1}, z^{-1}]^x\cdot [z, x^{-1}, y^{-1}]^z\cdot
[y,z^{-1},x^{-1}]^y = 1. $$

\subsection{General linear group}
Let $R$
be an associative ring with 1, $R^*$ be the multiplicative
group of the ring $R$. For two natural numbers $m,n$
we denote by $M(m,n,R)$ the additive group of $m\times n$-matrices
with entries in $R$. By $M(n,R)=M(n,n,R)$
we denote the full matrix ring of degree $n$ over $R$.
\par
Let $G=\GL(n,R)=M(n,R)^*$ be the general linear group of degree 
$n$ over $R$. In the sequel for a matrix
$g\in G$ we denote by $g_{ij}$ its matrix entry in the position
$(i,j)$, so that $g=(g_{ij})$, $1\le i,j\le n$. The inverse of
$g$ will be denoted by $g^{-1}=(g'_{ij})$, $1\le i,j\le n$.
\par
As usual we denote by $e$ the identity matrix of degree $n$
and by $e_{ij}$ a standard matrix unit, i.~e., the matrix
that has 1 in the position $(i,j)$ and zeros elsewhere.
An elementary transvection $t_{ij}(\xi)$ is a matrix of the
form $t_{ij}(c)=e+c e_{ij}$, $1\le i\neq j\le n$, $c\in R$. 
\par
Further, let $A$ be a two-sided of $R$. We consider the
corresponding reduction homomorphism
$$ \rho_A:\GL(n,R)\map\GL(n,R/A),\quad
(g_{ij})\mapsto(g_{ij}+A). $$
\noindent
Now, the {\it principal congruence subgroup} $\GL(n,R,A)$ of level $A$
is the kernel $\rho_A$, 
\par
For a commutative ring $R$ we denote by $\SL(n,R)$ the
corresponding general linear group. All other subgroups
are interpreted similarly. Thus, for instance, the principal 
congruence subgroup $\SL(n,R,A)$ is defined as 
$\SL(n,R,A)=\GL(n,R,A)\cap\SL(n,R)$.


\section{Generation of relative elementary subgroups}

The {\it unrelative elementary subgroup} $E(n,A)$ of level $A$
in $\GL(n,R)$ is generated by all elementary matrices of level 
$A$. In other words,
$$ E(n,A)=\langle e_{ij}(a),\ 1\le i\neq j\le n,\ a\in A \rangle. $$
\noindent
In general $E(n,A)$ has little chances to be normal in $\GL(n,R)$. 
The {\it relative elementary subgroup} $E(n,R,A)$  of level $A$
is defined as the normal closure of $E(n,A)$ in the absolute 
elementary subgroup $E(n,R)$:
$$ E(n,R,A)=\langle e_{ij}(a),\ 1\le i\neq j\le n,\ a\in A 
\rangle^{E(n,R)}. $$

The following lemma on generation of relative elementary subgroups
$E(n,R,A)$ is a classical result discovered in various contexts by 
Stein, Tits and Vaserstein, see, for instance, \cite{Vaserstein_normal}
(or \cite{Hazrat_Vavilov_Zhang}, Lemma 3, for a complete elementary
proof). It is stated in terms of the {\it Stein---Tits---Vaserstein 
generators\/}):
$$ z_{ij}(a,c)=t_{ij}(c)t_{ji}(a)t_{ij}(-c),\qquad
1\le i\neq j\le n,\quad a\in A,\quad c\in R. $$

\begin{lemma}\label{l2}
Let $R$ be an associative ring with $1$, $n\ge 3$, and let $A$ 
be a two-sided ideal of $R$. Then as a subgroup $E(n,R,A)$ is 
generated by $z_{ij}(a,c)$, for all $1\le i\neq j\le n$, $a\in A$, 
$c\in R$.
\end{lemma}

The following result is a generalisation of Lemma 2 to mutual commutator subgroups $[E(n,R,A),E(n,R,B)]$ of relative elementary subgroups. a further type of generators occur, the
{\it elementary commutators\/}:
$$ y_{ij}(a,b)=[t_{ij}(a),t_{ji}(b)],\qquad
1\le i\neq j\le n,\quad a\in A,\quad b\in B. $$
\noindent
In slightly less precise forms, Theorem ~A was discovered by 
Roozbeh Hazrat and the second author, see \cite{Hazrat_Zhang_multiple}, Lemma 12 and then in our joint 
paper with Hazrat \cite{Hazrat_Vavilov_Zhang}, Theorem~3A. 
The strong form reproduced above was only established in our
paper \cite{NZ2}, Theorem~1, as an aftermath of our papers
\cite{NV18, NZ1}.

\begin{lemma}\label{l3}
Let $R$ be any associative ring with $1$, let $n\ge 3$, and let $A,B$ 
be two-sided ideals of $R$. Then the mixed commutator subgroup 
$[E(n,R,A),E(n,R,B)]$ is generated as a group by the elements of the form
\par\smallskip
$\bullet$ $z_{ij}(ab,c)=t_{ij}(c)t_{ji}(ab)t_{ij}(-c)$ and 
$z_{ij}(ba,c)=t_{ij}(c)t_{ji}(ba)t_{ij}(-c)$,
\par\smallskip
$\bullet$ $y_{ij}(a,b)=[t_{ij}(a),t_{ji}(b)]$,
\par\smallskip\noindent
where $1\le i\neq j\le n$, $a\in A$, $b\in B$, $c\in R$. Moreover,
for the second type of generators, it suffices to fix one pair of
indices $(i,j)$.
\end{lemma}

Since all generators listed in Lemma 3 belong already to the
commutator subgroup of unrelative elementary subgroups, we
get the following corollary, \cite{NZ2}, Theorem~2.

\begin{lemma}\label{l4}
Let $R$ be any associative ring with $1$, let $n\ge 3$, and let $A,B$ 
be two-sided ideals of $R$.  Then one has
$$ \big[E(n,R,A),E(n,R,B)\big]=\big[E(n,R,A),E(n,B)\big]=\big[E(n,A),E(n,B)\big]. $$
\end{lemma}

Let us state also some subsidiary results we use in our proofs.
The following level computation is
standard, see, for instance, \cite{Vavilov_Stepanov_standard,
Vavilov_Stepanov_revisited, Hazrat_Vavilov_Zhang}, and references
there.

\begin{lemma}\label{l5}
$R$ be an associative ring with $1$, $n\ge 3$, and let $A$ and $B$
be two-sided ideals of $R$.  Then 
$$ E(n,R,A\circ B)\le\big[E(n,A),E(n,B)\big]\le
\big[E(n,R,A),E(n,R,B)\big] \le\GL(n,R,A\circ B). $$
\end{lemma}

However, when applying this lemma to multiple commutators, 
one should bear in mind that the symmetrised product is not 
associative. Thus, when writing something like $A\circ B\circ C$,
we have to specify the order in which products are formed.
Of course, for commutative rings this dependence on the
original bracketing disappears.

For quasi-finite rings the following result is \cite{Vavilov_Stepanov_revisited},
Theorem~5 and \cite{Hazrat_Vavilov_Zhang}, Theorem~2A, 
but for arbitrary associative rings it was only established in 
\cite{NZ3}, Theorem~2.

\begin{lemma}\label{l6}
Let $R$ be any associative ring with $1$, let $n\ge 3$, and let 
$A$ and $B$ be two-sided ideals of $R$. If $A$ and $B$ are
comaximal, $A+B=R$, then
$$ [E(n,A),E(n,B)]=E(n,R,A\circ B). $$
\end{lemma}


\section{Partially relativised elementary subgroups}

Actually, the recent work by Alexei Stepanov \cite{Stepanov_calculus, Stepanov_nonabelian, Stepanov_universal, AS} makes apparent 
that in many contexts it is very useful to consider {\it partially\/}
relativised subgroups. One such context is relative localisation, as
introduced in the papers by Roozbeh Hazrat and the second author
\cite{Hazrat_Zhang, Hazrat_Zhang_multiple}, then expanded and
 developed in a series of our joint papers with Hazrat, and 
reconsidered by Stepanov, see \cite{RNZ1, RNZ2, RNZ5, 
yoga-1, Porto-Cesareo, yoga-2, Hazrat_Vavilov_Zhang, Stepanov_nonabelian, Stepanov_universal}.

Namely, for two ideals $A,B\unlhd R$
we denote by $E(n,B,A)$ the smallest subgroup containing
$E(n,A)$ and normalised by $E(n,B)$:
$$ E(n,B,A)=E(n,A)^{E(n,B)}. $$
\noindent
In particular, when $B=R$ we get the usual relative group
$E(n,R,A)$, as defined above.
Clearly, if $B\le C$, then $E(n,B,A)\le E(n,C,A)$. It follows that
$$ E(n,A)=E(n,0,A)\le E(n,B,A)\le E(n,R,A) $$
\par
On the other hand, 
$$ \big[E(n,A),E(n,B)\big]\le
E(n,B,A)\cap E(n,A,B). $$
\noindent
Thus, Lemma 5 implies the following inclusion, which is a
broad generalisation of \cite{AS}, Lemma 4.1, in the linear case.

\begin{proposition}\label{Apte-Stepa}
Let $R$ be any associative ring with $1$, let $n\ge 3$, 
and let $A,B$ be two-sided ideals of $R$.  Then one has
$$ E(n,R,A\circ B)\le E(n,B,A)\cap E(n,A,B). $$
\end{proposition}

Now, we start working towards a partially relativised 
generalisation of Lemma 2.

\begin{lemma}\label{l7}
Let $R$ be any associative ring with $1$, let $n\ge 2$, and let $A,B$ 
be two-sided ideals of $R$.  Then one has
$$ \big\langle E(n,A),E(n,B)\big\rangle=E(n,A+B). $$
\end{lemma}
\begin{proof}
Clearly, the left hand side is contained in the right hand side. 
On the other hand $E(n,A+B)$ is generated by the elementary
transvections $t_{ij}(a+b)$, where $1\le i\neq j\le n$, $a\in A$,
$b\in B$. But every $t_{ij}(a+b)=t_{ij}(a)t_{ij}(b)\in E(n,A)E(n,B)$.
\end{proof}

In particular, even $E(n,A)E(n,B)=E(n,A+B)$, if the left hand side 
is a subgroup. But this is easy to remedy. Indeed, the above lemma 
implies that in the definition of partially relativised subgroups one can assume that $A\le B$.

\begin{Corollary}
Let $R$ be any associative ring with $1$, let $n\ge 2$, and let $A,B$ 
be two-sided ideals of $R$.  Then $E(n,B,A)=E(n,A+B,A)$.
\end{Corollary}

But since $E(n,B,A)$ is normalised by both $E(n,A)$ and $E(n,B)$,
it is normal in $E(n,A+B)$.

\begin{Corollary}
Let $R$ be any associative ring with $1$, let $n\ge 2$, and let $A,B$ 
be two-sided ideals of $R$.  Then 
$$ E(n,B,A)E(n,B)=E(n,A+B). $$
\end{Corollary}

Passing to the normal closures in $E(n,R)$ we get the familiar
equality, see, in particular, \cite{Vavilov_Stepanov_standard},
Lemma 1.

\begin{Corollary}
Let $R$ be any associative ring with $1$, let $n\ge 2$, and let $A,B$ 
be two-sided ideals of $R$.  Then 
$$ E(n,R,A)E(n,R,B)=E(n,R,A+B). $$
\end{Corollary}

The following result is a generalisation of a classical result on
generation of relative elementary subgroups $E(n,R,A)$, 
discovered in various contexts by Stein, Tits and Vaserstein, 
see, for instance, \cite{Vaserstein_normal}. It is stated in 
terms of the {\it Stein---Tits---Vaserstein generators\/}):
$$ z_{ij}(a,c)=t_{ij}(c)t_{ji}(a)t_{ij}(-c),\qquad
1\le i\neq j\le n,\quad a\in A,\quad c\in R. $$
\par
Essentially, its proof follows the proofs of Lemma 2 (as 
reproduced in \cite{Vavilov_Stepanov_1}, Theorem~1 
or \cite{Hazrat_Vavilov_Zhang}, Lemma 3, for instance). 
But of course a posteriori we can take advantage of the 
simplifications that result from Lemma 3.

\begin{theorem}\label{t2}
Let $R$ be an associative ring with identity $1$, $n\ge 3$,
and let $A$ and $B$ be two-sided ideals of $R$. Then the 
partially relativised elementary subgroup $E(n,B,A)$ is 
generated the elements $z_{ij}(a,b)$, for all $1\le i\neq j\le n$, 
$a\in A$, $b\in B$.
\end{theorem}

\begin{proof}
Clearly, $E(n,B,A)$ is generated by ${}^{y}x$ with $x\in E(n,A)$ 
and $y\in E(n,B)$. Since ${}^{y}x=[y,x]\cdot x$ we get 
$[y,x]\in [E(n,A), E(n,B)]$. By definition, the factor $x$  is 
a product of elementary matrices $t_{ij}(a)$ with $i\ne j$ 
and $a\in A$. By Lemma 3 the commutator subgroup 
$[E(n,A), E(n,B)]$ is generated by $E(n,R,A\circ B)$ together 
with the elementary commutators 
$$ y_{ij}(a,b)=[t_{ij}(a),t_{ji}(b)]=t_{ij}(a)\cdot 
{}^{t_{ji}(b)}t_{ij}(-a)=z_{ij}(a,0)z_{ij}(a,b), $$ 
\noindent
where $1\le i\ne j\le n$, $a\in A$ and $b\in B$.
\par
Thus, it only remains to show that the generators
$z_{ij}(ab,c)={}^{t_{ji}(c)}t_{ij}(ab)$ and 
$z_{ij}(ba,c)={}^{t_{ji}(c)}t_{ij}(ba)$ of the relative
elementary group $E(n,R,A\circ B)$ are products of 
generators listed in the statement of the lemma.
Here, as above, $1\le i\ne j\le n$, $a\in A$, $b\in B$ and
$c\in C$.
\par
We choose a $h\ne i,j$, then
\begin{multline*}
{}^{t_{ji}(c)}t_{ij}(ab)={}^{t_{ji}(c)}[t_{ih}(a), t_{hj}(b)]
=[{}^{t_{ji}(c)}t_{ih}(a), {}^{t_{ji}(c)}t_{hj}(b)]=\\
\big[[t_{ji}(c),t_{ih}(a)]t_{ih}(a), [t_{ji}(c),t_{hj}(b)]t_{hj}(b)\big]
=\big[t_{jh}(ca)t_{ih}(a), t_{hi}(-bc)t_{hj}(b)\big]=\\
{}^{t_{jh}(ca)}[t_{ih}(a), t_{hi}(-bc)t_{hj}(b)]\cdot 
[t_{jh}(ca), t_{hi}(-bc)t_{hj}(b)].
\end{multline*} 
To finish the proof, we consider the two above commutators separately 
$$ u=[t_{ih}(a), t_{hi}(-bc)t_{hj}(b)],\qquad 
v=[t_{jh}(ca), t_{hi}(-bc)t_{hj}(b)] $$ 
\noindent
The following computation shows that $u$ is a products of 
generators listed in the statement of the lemma.
\begin{multline*}
u=[t_{ih}(a), t_{hi}(-bc)t_{kj}(b)]
=[t_{ih}(a), t_{hi}(-bc)]\cdot {}^{t_{hi}(-bc)}[t_{ih}(a), t_{hj}(b)]=\\
[t_{ih}(a), t_{hi}(-bc)]\cdot {}^{t_{hi}(-bc)}t_{ij}(ab)
=[t_{ih}(a), t_{hi}(-bc)]\cdot [t_{hi}(-bc),t_{ij}(ab)]t_{ij}(ab)=\\
[t_{ih}(a), t_{hi}(-bc)]\cdot t_{hj}(-bcab)t_{ij}(ab)
=t_{ih}(a)\cdot {}^{t_{hi}(-bc)}t_{ih}(-a) t_{hj}(-bcab)t_{ij}(ab),
\end{multline*}
A similar computation shows that the same holds also for $v$,
which finishes the proof.
\end{proof}

\begin{lemma}\label{l8}
Let $R$ be an associative ring with identity $1$, and let $A$,  $B$,  $C$  and $D$ be its two-sided ideals.
Then we have the following commutator formula for partially 
relativised elementary subgroups
$$ [E(n,B,A), E(n,D,C)]=[E(n,A),E(n,C)]. $$
\end{lemma}
\begin{proof}
Combining the inclusions among partially relativised
subgroups with Lemma 4, we get
\begin{multline*}
[E(n,A),E(n,C)]=[E(n,0,A), E(n,0,C]]\le [E(n,B,A), E(n,D,C)]\le \\
[E(n,R,A), E(n,R,C)]
=[E(n,A),E(n,C)].
\end{multline*}
\end{proof}


\section{Parabolic subgroups}

In this section and the next one we collect some results 
on generation of $E(n,R,A)$ by the elements in unipotent 
radicals, or their conjugates. We start with the absolute case.

\subsection{Standard parabolic subgroups}
Denote by $R^n$ the free right $R$-module consisting of
columns of height $n$ with components from $R$. Similarly,
${}^n\!R$ denotes the free left $R$-module, consisting of
rows of length $n$ with components from $R$. The module ${}^n\! R$
is dual to $R^n$, with the pairing of ${}^n\! R$ and $R^n$
defined by the multiplication of a row by a column,
${}^n\! R\times R^n\map R$, $(v,u)\mapsto vu\in R$. The standard
based of $R^n$ and ${}^n\! R$ will be denoted by
$e_1,\ldots,e_n$ and $f_1,\ldots,f_n$, respectively.
Recall that $e_i$ is the column of height $n$, whose
$i$-th component equals 1, while all other components are
zeroes. Similarly, $f_i$ is the row of length $n$, whose
$i$-th component equals 1, while all other components are zeroes.
The base $f_1,\ldots,f_n$ is dual to $e_1,\ldots,e_n$,
with respect to the above pairing. The group $G=\GL(n,R)$
acts on $R^n$ on the left, by multiplication of a column
$u\in R^n$ by a matrix $g\in G$, $(g,u)\mapsto gu$.
By the same token, the group $G$ acts on ${}^n\! R$ on the
right by multiplication: $(v,g)\mapsto vg$ for $v\in{}^n\! R$,
$g\in G$.
\par
Denote by $P_m$ the $m$-th standard {\it maximal parabolic
subgroup\/} in $G=\GL(n,R)$. From a geometric viewpoint the
subgroup $P_i$, $m=1,\ldots,n-1$, is precisely the stabiliser
of the submodule $V_m$ in $V$, generated by
$e_1,\ldots,e_m$. In matrix form $P_m$ can be thought of as the
group of upper block triangular matrices
$$ P_m=\left\{
\begin{pmatrix} x&y\\ 0&z\\ \end{pmatrix} \mid x\in\GL(m,R), y\in
M(m,n-m,R), z\in\GL(n-m,R)\right\}. $$
\noindent
Simultaneously we consider the {\it opposite\/} maximal parabolic
subgroup $P_m^-$
$$ P_m^-=\left\{
\begin{pmatrix} x&0\\ w&z\\ \end{pmatrix} \mid x\in\GL(m,R), w\in
M(n-m,m,R), z\in\GL(n-m,R)\right\}. $$
\par
These subgroups admit Levi decompositions 
$P_m=L_m\rightthreetimes U_m$
and $P_m^-=L_m\rightthreetimes U_m^-$ with common Levi subgroup
$$ L_m=\left\{
\begin{pmatrix} x&y\\ 0&z\\ \end{pmatrix} \mid x\in\GL(i,R), 
z\in\GL(n-m,R)\right\}, $$
\noindent
and opposite unipotent radicals 
$$ U_m=\left\{
\begin{pmatrix} e&y\\ 0&e\\ \end{pmatrix} \mid y\in
M(m,n-m,R) \right\},\quad 
U_m^-=\left\{
\begin{pmatrix} e&0\\ w&e\\ \end{pmatrix} \mid w\in
M(n-m,m,R) \right\}. $$
\noindent
In particular, $L_m$, and thus all of its subgroups, normalise both
$U_m$ and $U_m^-$.

\subsection{Generation by two opposite unipotent radicals}
The following result asserts that $E(n,R)$ is generated by 
the unipotent radicals of two standard parabolic subgroups.
It is obvious from the Chevalley commutator formula, and well 
known.
Actually, in more general settings this is the {\it definition\/} 
of elementary subgroups, see the paper by Victor Petrov and
Anastasia Stavrova \cite{PS} and references 
there\footnote{Of course, with this advanced approach one 
has to prove that this definition is correct, in other words 
that various parabolic subgroups lead to the same elementary 
subgroup. This is precisely what is accomplished in \cite{PS}.}.

\begin{lemma}\label{l9}
Let $R$ be an associative ring with $1$, $n\ge 3$, and $1\le m\le n-1$. 
Then 
$$ E(n,R)=\langle U_m,U_m^-\rangle. $$
\end{lemma}

What is important, is that both generators here are normalised
by the Levi subgroup $L_m$ and all of its subgroups.
\par
As usual, for $m<n$ we consider the stability embedding 
$$ \GL(m,R)\map\GL(n,R),\qquad 
g\mapsto \begin{pmatrix} g&0\\ 0&e\\ \end{pmatrix}. $$
\noindent
This embedding is compatible with elementary subgroups, 
congruence subgroups, relative elementary subgroups, etc.
When we consider $\GL(m,R,A)$, etc., as a subgroup of
$\GL(n,R)$ we always mean its image under this embedding.
\par

\par
\subsection{Unipotent radicals of level $A$}
Now, let $A\unlhd R$ be an ideal of $R$. We denote by $U_m(A)$ and $U_n^-(A)$ the intersections
of $U_m$ and $U_m^-$ with $\GL(n,R,A)$:
\begin{align*}
&U_m(A)=\left\{\begin{pmatrix} e&y\\ 0&e\\ \end{pmatrix}\mid
y\in M(m,n-m,A)\right\},\\
&U_m^-(A)=\left\{\begin{pmatrix} e&0\\ w&e\\ \end{pmatrix}\mid
w\in M(n-m,m,A) \right\}. 
\end{align*}

The following lemma is a direct corollary of the 
Levi decomposition for
$P_{m}$ and its opposite $P^-_{m}$.

\begin{lemma}\label{l10}
Let $R$ be an associative ring with $1$, $n\ge 3$, $1\le m\le n-1$, 
and let $A,C$ be two-sided ideals of $R$. Then one has
$$ [\GL(m,R,A),U_m(C)]\le U_m(AC),\qquad [
\GL(m,R,A),U_m^-(C)]\le U_m^-(CA). $$
\end{lemma}

\begin{proof} Let $g\in\GL(m,R,A)$, $y\in M(m,n-m,C)$, and
$w\in M(n-m,m,C)$. Then, clearly,
$$ \left[ \begin{pmatrix} g&0\\ 0&e\\ \end{pmatrix},
\begin{pmatrix} e&y\\ 0&e\\ \end{pmatrix}\right]=
\begin{pmatrix} e&(g-e)y\\ 0&e\\ \end{pmatrix},\quad
\left[ \begin{pmatrix} g&0\\ 0&e\\ \end{pmatrix},
\begin{pmatrix} e&0\\ w&e\\ \end{pmatrix}\right]=
\begin{pmatrix} e&0\\ w(g^{-1}-e)&e\\ \end{pmatrix}, $$
\noindent
where $y\equiv 0\pamod{C}$, $w\equiv 0\pamod{C}$ and
$g\equiv g^{-1}\equiv e\pamod{A}$. Then all outer-diagonal
entries of the matrices in the right hand sides are congruent
to $0$ modulo $AC$, as claimed.
\end{proof}


\section{Limiting the set of generators for $E(n,R,A)$}

Let us state a result by Wilberd van der Kallen, \cite{vdK-group}, 
Lemma 2.2. Morally, it is a trickier and mightier version of Lemma 1, 
with a smaller set of generators.

\begin{lemma}\label{l11}
Let $A\unlhd R$ be an ideal of an associative ring, $n\ge 3$. Then
as a subgroup $E(n,R,A)$ is generated by $E(n,A)$ and 
$z_{in}(a,d)$, for all $1\le i\le n-1$, $a\in A$, $d\in R$ 
\end{lemma}

Generalisations of this result to unipotent radicals of parabolics
in arbitrary Chevalley groups were obtained by Alexei Stepanov in
\cite{Stepanov_calculus, Stepanov_nonabelian, Stepanov_universal},
(of course, in these papers $R$ was assumed commutative).

Below we extract the rationale behind these results by van der 
Kallen and Stepanov and prove a still stronger version of their
results, for the case of $\GL(n,R)$. Formally, it is not necessary
for the rest of the paper, Lemma 11 would suffice. Yet, it is so 
natural in itself, that it will be certainly prove useful for future
applications. Also, it is a midget version of our main result, Theorem~1,
with one ideal instead of three, and with $y_{ij}(ab,c)$ replaced 
by $z_{ij}(a,d)$. Actually, in the next section we breed it up to a 
life-size toy version of our main results. Eventually, we believe 
it should be taken as the correct definition of relative elementary 
groups in more general settings, viz., for isotropic reductive 
groups, see \cite{PS, Stavrova-Stepanov}.
\par
First, we reproduce the proof of Lemma 11 for $\GL(3,R)$.

\begin{lemma}\label{l12}
Let $A\unlhd R$ be an ideal of an associative ring, $n\ge 3$, and
let $i,j,h$ be three pair-wise distinct indices. Then if a subgroup
$E(n,A)\le H\le\GL(n,R)$ contains 
\par\smallskip
$\bullet$ either $z_{ih}(a,d)$ and $z_{jh}(a,d)$,
\par\smallskip
$\bullet$ or $z_{hi}(a,d)$ and $z_{hj}(a,d)$,
\par\smallskip\noindent
for all $a\in A$, $d\in R$, then it also contains $z_{ij}(a,d)$ 
and $z_{ji}(a,d)$, for all such $a$ and $d$.
\end{lemma}
\begin{proof}
We prove the desired inclusions for the first case, the second one
is treated similarly. Indeed,
$$ z_{ij}(a,d)={}^{t_{ji}(d)}t_{ij}(a)={}^{t_{ji}(d)}[t_{ih}(a),t_{hj}(1)]=
[t_{ih}(a)t_{jh}(da),t_{hj}(1)t_{hi}(-d)]. $$
\noindent
Expanding the commutator with respect to the second factor,
we see that 
$$ z_{ij}(a,d)=[t_{ih}(a)t_{jh}(da),t_{hj}(1)]\cdot
{}^{t_{hj}(1)}[t_{ih}(a)t_{jh}(da),t_{hi}(-d)]. $$
\noindent
Now, expanding both commutators with respect to the first factor,
we see that 
\begin{multline*}
z_{ij}(a,d)={}^{t_{jh}(da)}[t_{ih}(a),t_{hj}(1)]\cdot
[t_{jh}(da),t_{hj}(1)]\cdot\\
{}^{t_{hj}(1)t_{ih}(a)}[t_{jh}(da),t_{hi}(-d)]\cdot
{}^{t_{hj}(1)}[t_{ih}(a),t_{hi}(-d)].
\end{multline*}
\noindent
Consider the four factors on the right hand side individually
\par\smallskip
$\bullet$ ${}^{t_{jh}(da)}[t_{ih}(a),t_{hj}(1)]=t_{ij}(a)t_{ih}(-ada)$,
\par\smallskip
$\bullet$ $[t_{jh}(da),t_{hj}(1)]=t_{jh}(da)\cdot z_{jh}(-da,1)$,
\par\smallskip
$\bullet$ ${}^{t_{hj}(1)t_{ih}(a)}[t_{jh}(da),t_{hi}(-d)]=
t_{ji}(-dad)t_{hi}(-dad)\cdot z_{jh}(dada,1)$,
\par\smallskip
$\bullet$ ${}^{t_{hj}(1)}[t_{ih}(a),t_{hi}(-d)]=
t_{ih}(a)t_{ij}(-a)\cdot {}^{t_{hj}(1)}z_{ih}(-a,-d)$, but by Theorem~A 
the last factor only differs from $z_{ih}(-a,-d)\in\GL(2,R,A)$ by a factor 
from $E(n,A)$.
\par\smallskip
We see that all factors only involve matrices from $E(n,A)$ and 
elementary conjugates of the form $z_{ih}(b,c)$ and $z_{jh}(b,c)$,
for some $b\in A$, $c\in R$, as claimed.
\end{proof}

\begin{theorem}\label{t3}
Let $R$ be an associative ring with identity $1$, $n\ge 3$,
and let $A$ be a two-sided ideal of $R$. Fix an $m$,
$1\le m\le n-1$. Then the relative elementary subgroup 
$E(n,R,A)$ is generated by 
$$ U_m^-(A)\qquad\text{and}\qquad 
uvu^{-1},\quad\text{where}\quad v\in U_m(A), u\in U_m^-. $$
\end{theorem}
\begin{proof}
Let $H$ be the group generated by the above elements. 
First, observe that $E(n,A)\le H$. Indeed, the following case analysis 
shows that $H$ contains all generators of $E(n,A)$:
\par\smallskip
$\bullet$
When $i\le m$ and $j\ge m+1$,
or when $i\ge m+1$ and $j\le m$, the generator 
$t_{ij}(a)=z_{ij}(a,0)$ belongs to $H$ by assumption. 
\par\smallskip
$\bullet$ When 
$i,j\le m$ take any $h\ge m+1$. Then
$t_{ij}(a)=t_{ih}(a)\cdot{}^{t_{hj}(1)}t_{ih}(-a)\in H$. 
\par\smallskip
$\bullet$ When $i,j\ge m+1$ take any $h\le m$. Then
$t_{ij}(a)={}^{t_{ih}(1)}t_{hj}(a)\cdot t_{hj}(-a)\in H$.
\par\smallskip
Now, we are done by repeated application of Lemma 12. Indeed,
$z_{ij}(a,d)\in H$ by assumption when $i\le m$, $j\ge m+1$.
\par\smallskip
$\bullet$ When $i,j\le m$ take any $h\ge m+1$. Then
$z_{ih}(a,d), z_{jh}(a,d)\in H$ and thus $z_{ij}(a,d),z_{ji}(a,d)\in H$
by the first item of Lemma 12.
\par\smallskip
$\bullet$ When $i,j\ge m+1$ take any $h\le m$. Then
$z_{hi}(a,d), z_{hj}(a,d)\in H$ and thus $z_{ij}(a,d),z_{ji}(a,d)\in H$
by the second item of Lemma 12.
\par\smallskip
Finally, when $i\ge m+1$ and $j\le m$, one has to distinguish two 
cases.
\par\smallskip
$\bullet$ If $m\ge 2$, one can choose $h\le m$, $h\neq j$. Then
$z_{hj}(a,d)\in H$ by assumption, whereas $z_{hi}(a,d)\in H$
by the first item above. Thus, $z_{ij}(a,d)\in H$ by the second item 
of Lemma 12.
\par\smallskip
$\bullet$ If $m\le n-2$, one can choose $h\ge m+1$, $h\neq j$.  Then
$z_{ih}(a,d)\in H$ by assumption, whereas $z_{jh}(a,d)\in H$
by the second item above. Thus, $z_{ij}(a,d)\in H$ by the first item 
of Lemma 12. 
\par\smallskip\noindent
It remains only to refer to Lemma~2 --- or Theorem~2, for that matter.
\end{proof}

\begin{corollary}
Let $R$ be an associative ring with identity $1$, $n\ge 3$,
and let $A$ be a two-sided ideal of $R$. Fix an $m$,
$1\le m\le n-1$. Then the relative elementary subgroup 
$E(n,R,A)$ is generated by the group $E(n,A)$ and the elements 
$z_{ij}(a,d)$, for all $i\neq j$, $1\le i\le m$, $m+1\le j\le n$.
$a\in A$, $d\in R$.
\end{corollary}


\section{Three ideals lemma for $E(n,C,AB)$}

It is natural to ask, whether Theorem~2 admits a similar stronger
version. Unfortunately, the above proof of Theorem~3 does not 
generalise immediately to the partially relativised case.

\begin{problem}
Let $R$ be an associative ring with identity $1$, $n\ge 3$,
and let $A$, $B$ be a two-sided ideals of $R$. Fix an $m$,
$1\le m\le n-1$. Is the partially relativised elementary subgroup 
$E(n,B,A)$ generated by 
$$ U_m^-(A)\qquad\text{and}\qquad 
uvu^{-1},\quad\text{where}\quad v\in U_m(A), u\in U_m^-(B)? $$
\end{problem}

Instead, we prove the following refinement of Lemma~12, which
gives a full scale generalisation of Proposition 1, and a toy version
of our Theorems~1 and~6.

\begin{theorem}\label{t3bis}
Let $R$ be an associative ring with identity $1$, $n\ge 3$,
and let $A,B,C$ be a two-sided ideals of $R$. Then $E(n,C,AB)$
is contained in any of the following three spans:
\begin{multline*}
\big\langle E(n,BC,A), E(n,B,CA)\big\rangle,\qquad
\big\langle E(n,A,BC), E(n,CA,B)\big\rangle,\\
\big\langle E(n,BC,A), E(n,CA,B)\big\rangle. 
\end{multline*}
\end{theorem}

\begin{proof}
By Theorem~2 the group $E(n,C,AB)$ is generated by the
elementary commutators $z_{ij}(ab,c)$, where $a\in A$,
$b\in B$, $c\in C$. Now we imitate the proof of Lemma 12, but now 
monitor the levels of the occurring parameters of the $z_{rs}$'s in
the right hand side, 
rather than their positions. As in Lemma 12 we take and $h\neq i,j$
and rewrite the generator $z_{ij}(c,ab)$ as a commutator:
$$ z_{ij}(ab,c)={}^{t_{ji}(c)}t_{ij}(ab)=
{}^{t_{ji}(c)}[t_{ih}(a),t_{hj}(b)]=
[t_{ih}(a)t_{jh}(ca),t_{hj}(b)t_{hi}(-bc)]. $$
\noindent
Expanding the commutator with respect to the second factor,
we see that 
$$ z_{ij}(ab,c)=[t_{ih}(a)t_{jh}(ca),t_{hj}(b)]\cdot
{}^{t_{hj}(b)}[t_{ih}(a)t_{jh}(ca),t_{hi}(-bc)]. $$
\noindent
Now, expanding both commutators with respect to the first factor,
we see that 
\begin{multline*}
z_{ij}(ab,c)={}^{t_{jh}(ca)}[t_{ih}(a),t_{hj}(b)]\cdot
[t_{jh}(ca),t_{hj}(b)]\cdot\\
{}^{t_{hj}(b)t_{ih}(a)}[t_{jh}(ca),t_{hi}(-bc)]\cdot
{}^{t_{hj}(b)}[t_{ih}(a),t_{hi}(-bc)].
\end{multline*}
\noindent
Consider the four factors on the right hand side individually.
Observe that by Lemma 4 the commutator subgroup
$[E(n,A),E(n,B)]$ and other such double commutators are
normal in $E(n,R)$, so that we can ignore all occurring
elementary conjugations.
Amazingly, the only problematic factor is the first one!
\par\smallskip
$\bullet$ Clearly, ${}^{t_{jh}(ca)}[t_{ih}(a),t_{hj}(b)]=
t_{ij}(ab)t_{ih}(-abca)$ belongs to 
$$ E(n,AB)\le E(n,A)\cap E(n,B)\le E(n,BC,A)\cap E(n,CA,B). $$
\par
$\bullet$ Further, $[t_{jh}(ca),t_{hj}(b)]$ belongs to 
$$ [E(n,CA),E(n,B)]\le E(n,CA,B)\cap E(n,B,CA). $$
\par
$\bullet$ Next, ${}^{t_{hj}(b)t_{ih}(a)}[t_{jh}(ca),t_{hi}(-bc)]$
belongs to 
\begin{multline*}
[E(n,CA),E(n,BC)]\le E(n,CA,BC)\cap E(n,BC,CA)\le\\
E(n,A,BC)\cap E(n,B,CA)\cap E(n,BC,A)\cap E(n,CA,B).
\end{multline*}
\par
$\bullet$ Finally, ${}^{t_{hj}(b)}[t_{ih}(a),t_{hi}(-bc)]$
belongs to 
$$ [E(n,A),E(n,BC)]\le E(n,A,BC)\cap E(n,BC,A). $$
\par
We see that the third factor belongs to all four subgroups,
and can be discarded, whereas the other three factors are 
contained in two of the subgroups $E(n,BC,A)$, $E(n,B,CA)$, 
$E(n,A,BC)$, $E(n,CA,B)$, each. Inspecting the cases
listed in the statement, we see that all of them contain
all three factors.
\end{proof}

The other three possible combinations of the subgroups 
$E(n,BC,A)$, $E(n,B,CA)$, 
$E(n,A,BC)$, $E(n,CA,B)$, do not seem to work in general. 
Thus, for instance, $\big\langle E(n,A,BC), E(n,B,CA)\big\rangle$
does not contain the first factor, and so on.


\section{Stable version of the standard commutator formula}

We start with a slightly more general form of \cite{NZ2}, 
Lemma~3 and \cite{NZ3}. Lemma~9. Essentially, it is a
classical corollary of surjective  stability for $K_1$, but again 
we need a birelative version.

The following lemma is what stays behind \cite{NZ2}, 
Lemma 3, and \cite{NZ3}, Lemma 9. Our argument here is
both much more general, and much easier, since it avoids 
all explicit computations.

\begin{lemma}\label{l13}
Let $R$ be an associative ring with $1$, $n\ge 3$, and let $A$ 
be a two-sided ideal of $R$. Then for any  $g\in\GL(n-1,R,A)$ 
and any $x\in E(n,R)$ one has 
$$ {}^xg\equiv g\pamod{E(n,R,A)}. $$
\end{lemma}

\begin{proof}
By Lemma~\ref{l9} any $x\in E(n,R)$ can be expressed as a product
$x=y_1\ldots y_m$, where $y_i$ alternatively belong to 
$U_{n-1}$ or $U^-_{n-1}$. Consider a shorts such product. We argue 
by induction on $m$. 
\par
Let $x=yz$, where $y\in E(n,R)$ is shorter than $x$, whereas
$z\in U_{n-1}$ or $z\in U^-_{n-1}$.  By Lemma~\ref{l10}
$[g,z]\in U_{n-1}(A)$ in the first case, and 
$[g,z]\in U_{n-1}^-(A)$ in the first case. Since $U_{n-1}(A),
U_{n-1}^-(A)\le E(n,A)\le E(n,R,A)$, this means that 
${}^zg\equiv g\pamod{E(n,R,A)}$.
This means that ${}^xg\equiv {}^yg\pamod{E(n,R,A)}$.
But ${}^yg\equiv g\pamod{E(n,R,A)}$ by induction hypothesis.
\end{proof}

Of course, one would love to have a similar {\it birelative} lemma,
asserting that for any $g\in\GL(n-1,R,A)$ and any $x\in E(n,R,C)$ 
one has ${}^xg\equiv g\pamod{E(n,R,A)}$. This would give 
plenty of leverage, to establish very strong results, including 
Theorem~1, with minimum direct calculations. 

Unfortunately, it seems that such a lemma does not hold. What 
we can see easily, is only the weaker congruence 
${}^xg\equiv g\pamod{[E(n,A),E(n,C)]}$. The following
result is a version of the standard commutator formula that
survives for arbitrary associative rings. Various forms of this
result are known for decades, since the groundbreaking paper 
by Hyman Bass \cite{Bass_stable}, and the refinements by Alec
Mason and Wilson Stothers \cite{Mason_Stothers}, see our 
exposition in \cite{Hazrat_Vavilov_Zhang}. However, the proofs
proceeded as follows. First, one established a more sophisticated 
double relative version of Whitehead lemma, and then invoked
deep results, such as Bass---Vaserstein injective stability for 
$\K_1$. Our proof below is entirely elementary, works for all
associative rings, and only uses the sharp generation results 
obtained in the previous sections.

\begin{theorem}\label{t4}
Let $R$ be an associative ring with $1$, $n\ge 3$, and let $A$ 
and $C$ be two-sided ideals of $R$. Then 
$$ [\GL(n-1,R,A),E(n,R,C)]=[E(n,A), E(n,C)]. $$
\end{theorem}

\begin{proof}
Indeed, by Theorem~\ref{t3} the group $E(n,R,C)$ is generated by 
$w\in U_{n-1}^-(C)$ and by
$uvu^{-1}$, where $v\in U_{n-1}(C)$, $u\in U_{n-1}^-$. Take an arbitrary
$g\in\GL(n-1,R,A)$. Then $[g,w]\in E(n,CA)$ by Lemma 10. On the
other hand, for the other type of generators one has
$$ [g,uvu^{-1}]=[g,u]\cdot {}^u[g,v]\cdot {}^{uv}[g,u^{-1}]=
[g,u]\cdot {}^u[g,v]\cdot {}^{u}[v,[g,u^{-1}]]\cdot {}^u[g,u^{-1}]. $$
\noindent
Now, by Lemma 11 one has $[g,v]\in E(n,AC)$, so that 
${}^u[g,v]\in E(n,R,AC)$. 
Similarly, $[g,u^{-1}]\in E(n,A)$, so that
$[v,[g,u^{-1}]]\in [E(n,A),E(n,C)]$. It follows from Lemma 4
(but was known before, in fact), that $[E(n,A),E(n,C)]$ 
is normal in $E(n,R)$. Thus, ${}^{u}[v,[g,u^{-1}]]\in [E(n,A),E(n,C)]$.
\par
By Lemma 5, one has $E(n,R,AC)\le [E(n,A),E(n,C)]$ so that both
central factors belong to $[E(n,A),E(n,C)]$. On the other hand,
${}^u[g,u^{-1}]=[g,u]^{-1}$. Again invoking the fact that
$[E(n,A),E(n,C)]$ is normal in $E(n,R)$ we see that the commutator
$[g,uvu^{-1}]$ belongs to $[E(n,A),E(n,C)]$. 
\par
Since the elements $uvu^{-1}$ generate $E(n,R,C)$ and are
themselves elementary, the left hand side of the equality in the
statement of the theorem is contained in the right hand side. The
other inclusion is obvious.
\end{proof}

Of course, when surjective stability holds for $\K_1(n-1,R,A)$, one
has 
$$ \GL(n,R,A)=\GL(n-1,R,A)E(n,R,A), $$
\noindent
so that Theorem~4 implies the usual standard commutator formula
$$ [\GL(n,R,A),E(n,R,C)]=[E(n,A), E(n,C)]. $$
\noindent
Otherwise, we use Theorem~4 in the following form.

\begin{corollary}
Let $R$ be an associative ring with $1$, $n\ge 3$, and let $A$ 
and $C$ be two-sided ideals of $R$. Then for any  $g\in\GL(n-1,R,A)$ 
and any $x\in E(n,R,C)$ one has 
$$ {}^xg\equiv g\pamod{[E(n,A), E(n,C)]}. $$
\end{corollary}


\section{Elementary commutators modulo $E(n,R,A\circ B)$}

In the present section we collect special cases of the previous 
results concerning the behaviour of elementary commutators 
modulo the level. 

Since the elementary commutator $y_{ij}(a,b)$, where 
$1\le i\neq j\le n$, $a\in A$, $b\in B$, has level $A\circ B$,
we get the following result, which is \cite{NZ2}, 
Lemma 3, and \cite{NZ3}, Lemma 9. 

\begin{lemma}\label{l14}
Let $R$ be an associative ring with $1$, $n\ge 3$, and let $A,B$ 
be two-sided ideals of $R$. Then for any  $1\le i\neq j\le n$, 
$a\in A$, $b\in B$, and any $x\in E(n,R)$ one has
$$ {}^x y_{ij}(a,b)\equiv y_{ij}(a,b) 
\pamod{E(n,R,A\circ B)}. $$
\end{lemma}

It is well known that the absolute elementary group $E(n,R)$
contains all permutation matrices, maybe after correcting the
sign of one entry. Thus, already this lemma implies that 
elementary commutators $y_{ij}(a,b)$ and $y_{hk}(a,b)$
are congruent modulo $E(n,R,A\circ B)$. Of course, we still
need Theorem~A, since we need to move around not only
the indices, but also the parameters. 

The following result is \cite{NZ3}, Lemmas 10 and 12.

\begin{lemma}\label{l15}
Let $R$ be an associative ring with $1$, $n\ge 3$, and let $A,B$ 
be two-sided ideals of $R$. Then for any  $1\le i\neq j\le n$, 
$a,a_1,a_2\in A$, $b,b_1,b_2\in B$ one has
\begin{align*}
&y_{ij}(a_1+a_2,b)\equiv  y_{ij}(a_1,b)\cdot y_{ij}(a_1,b) 
\pamod{E(n,R,A\circ B)},\\
&y_{ij}(a,b_1+b_2)\equiv  y_{ij}(a,b_1)\cdot y_{ij}(a,b_2) 
\pamod{E(n,R,A\circ B)},\\
&y_{ij}(a,b)^{-1}\equiv  y_{ij}(-a,b)\equiv y_{ij}(a,-b) 
\pamod{E(n,R,A\circ B)},\\
&y_{ij}(ab_1,b_2)\equiv y_{ij}(a_1,a_2b)\equiv e
\pamod{E(n,R,A\circ B)},\\
&y_{ij}(a_1a_2,b)\equiv y_{ij}(a,b_1b_2)\equiv e
\pamod{E(n,R,A\circ B)}. 
\end{align*}
\end{lemma}

Together with Theorem~A this lemma asserts that modulo
$E(n,R,A\circ B)$ the elementary commutators $y_{ij}(a,b)$
do not depend on the choice of a pair $(i,j)$, $i\neq j$, and
can be considered as symbols
\begin{align}
\sigma:A/A(A+B)\otimes_R B/B(A+B)\map {}&[E(n,A),E(n,B)]/E(n,R,A\circ B),\\
(a+A(A+B))\otimes (b+B(A+B))\mapsto {}&y_{12}(a,b)\pamod{E(n,R,A\circ B)}. 
\end{align}

Let us reiterate \cite{NZ2}, Problem 1, and \cite{NZ3}, Problem 2.

\begin{problem}
Give a presentation of
$$ \big[E(n,A),E(n,B)\big]/E(n,R,A\circ B) $$
\noindent
by generators and relations. Does this presentation depend on 
$n\ge 3$? 
\end{problem}

The following lemma is classically known and obvious. It follows 
from the fact that in a Dedekind ring $R$ for any two ideals $A$ 
and $B$ there exists an ideal $C\cong A$ such that $C+B=R$.

\begin{lemma}\label{l16}
Let $R$ be a Dedekind ring, $A$ and $B$ be ideals of $R$. Then
$$ A/A(A+B)\cong B/B(A+B)\cong R/(A+B). $$
\end{lemma}

Thus, in this case the above symbols $\sigma$ can be considered as 
symbols
$$ \sigma:R/(A+B)\otimes_R R/(A+B)\map 
{}[E(n,A),E(n,B)]/E(n,R,A\circ B), $$
\noindent
closely related to the usual Mennicke symbols. We intend to address
Problem 2 for Dedekind rings in a subsequent paper.


\section{Proof of Theorem~1}

Now, we are all set to prove the technical heart of the present
paper, Theorem~1.

\begin{proof}
We take any $ h\neq i,j$ and rewrite the elementary commutator
$ y_{ij}(ab,c)=\big[t_{ij}(ab),t_{ji}(c)\big]$ as
$$  y_{ij}(ab,c)=t_{ij}(ab)\cdot{}^{t_{ji}(c)}t_{ij}(-ab)=
t_{ij}(ab)\cdot{}^{t_{ji}(c)}\big[t_{ih}(a),t_{hj}(-b)\big]. $$
\noindent
Expanding the conjugation by $t_{ji}(b)$, we see that 
$$  y_{ij}(ab,c)=t_{ij}(ab)\cdot\big[{}^{t_{ji}(c)}t_{ih}(a),{}^{t_{ji}(c)}t_{hj}(-b)\big] = 
t_{ij}(ab)\cdot[t_{jh}(ca)t_{ih}(a),t_{hj}(-b)t_{hi}(bc)\big]. $$
\noindent
Expanding the commutator in the right hand side, using 
multiplicativity of the commutator w.r.t. the second argument,
we get
$$  y_{ij}(ab,c)=t_{ij}(ab)\cdot
\big[t_{jh}(ca)t_{ih}(a),t_{hj}(-b)\big]\cdot
{}^{t_{hj}(-b)}\big[t_{jh}(ca)t_{ih}(a),t_{hi}(bc)\big]. $$
\noindent
Expanding the first commutator in the right hand side, and using
multiplicativity of the commutator w.r.t. the first argument,
we get
\begin{multline*}
\big[t_{jh}(ca)t_{ih}(a),t_{hj}(-b)\big]=
{}^{t_{jh}(ca)}\big[t_{ih}(a),t_{hj}(-b)\big]\cdot
\big[t_{jh}(ca),t_{hj}(-b)\big]=\\
t_{ij}(-ab)\cdot t_{ih}(abca)\cdot y_{jh}(ca,-b)
\end{multline*}
\noindent
Now, the first factor cancels with $t_{ij}(ab)$, the second
factor belongs to $E(n,ABC)$, and can be discarded, so that
the first commutator is congruent modulo $E(n,R,ABC)$ to
$y_{jh}(ca,-b)$. By Lemma~14 one has 
$$ y_{jh}(ca,-b)\equiv y_{jh}(ca,b)^{-1}
\pamod{E(n,R,CA\circ B)}. $$
\par
Next, we look at the second commutator in the right hand side
of the formula for $y_{ij}(ab,c)$, and using 
multiplicativity of the commutator w.r.t. the first argument,
we get
\begin{multline*}
{}^{t_{hj}(-b)}\big[t_{jh}(ca)t_{ih}(a),t_{hi}(bc)\big]=
{}^{t_{hj}(-b)t_{jh}(ca)}\big[t_{ih}(a),t_{hi}(bc)\big]\cdot
{}^{t_{hj}(-b)}\big[t_{jh}(ca),t_{hi}(bc)\big]=\\
{}^{t_{hj}(-b)t_{jh}(ca)}y_{ih}(a,bc)\cdot
{}^{t_{hj}(-b)}t_{ji}(cabc).
\end{multline*}
\noindent
Now, the second factor belongs to $E(n,ABC)$, and stays 
there after an elementary conjugation, so it can be 
discarded. The first factor is congruent to $y_{ih}(a,bc)$ 
modulo $E(n,R,A\circ BC)$ by Lemma~\ref{l14}. 
Again by Lemma~\ref{l14} one has
$$ y_{ih}(a,bc)\equiv y_{hi}(bc,-a)\equiv y_{hi}(bc,a)^{-1}
\pamod{E(n,R,A\circ BC)}. $$
\par
Summarising the above, we see that
$$ y_{ij}(ab,c) y_{jh}(ca,b) y_{hi}(bc,a)\equiv 1
\pamod{E(n,R,ABC+BCA+CAB)}, $$
\noindent
as claimed. Observe that since all factors are central in $E(n,R)$ 
modulo the normal subgroup $E(n,R,ABC+BCA+CAB)$, which
equals
$$ E(n,R,A\circ BC)\cdot E(n,R,B\circ CA)\cdot E(n,R,C\circ AB), $$
\noindent
their order does not matter.
\end{proof}

By the last remark in the proof of Theorem~1, the levels of all three commutators in the next result are contained in the normal subgroup $E(n,R,ABC+BCA+CAB)$. Thus, Theorem~1 immediately implies
the following result that can be interpreted as a {\it three ideals
lemma\/}.

\begin{theorem}\label{t5}
Let $R$ be an associative ring with $1$, $n\ge 3$, and let $A,B,C$ 
be two-sided ideals of $R$. Then
$$ [E(n,AB),E(n,C)] \le [E(n,BC),E(n,A)]\cdot [E(n,CA),E(n,B)]. $$
\end{theorem}

Modulo Lemma~\ref{l14} on triple commutator subgroups, it becomes
a special case of the three subgroups lemma, at least in the {\it commutative\/} case. However, the proof of Lemma~\ref{l14}  itself
crucially depends on a version of Theorem~A, Theorem~1, or
a similar calculation.


\section{Applications to multiple commutators}

Our proof of elementary multiple commutator formulas in
\cite{NZ3} is an easy induction that proceeds from the 
following two special cases, triple commutators, and quadruple 
commutators. The following results are \cite{NZ3}, Lemma 7
and Lemma 8, respectively.

\begin{lemma}\label{l17}
Let $R$ be an associative ring with $1$, $n\ge 3$, and let $A,B,C$ 
be two-sided ideals of $R$. Then
$$ \big[\big[E(n,A),E(n,B)\big],E(n,C)\big]=
\big[E(n,A\circ B),E(n,C)\big]. $$
\end{lemma}

\begin{lemma}\label{l18}
Let $R$ be an associative ring with $1$, $n\ge 4$, and let $A,B,C,D$ 
be two-sided ideals of $R$. Then
$$ \big[\big[E(n,A),E(n,B)\big],\big[E(n,C),E(n,D)\big]\big]=
\big[E(n,A\circ B),E(n,C\circ D)\big]. $$
\end{lemma}

These results were first proven for quasi-finite rings by
Roozbeh Hazrat and the second author, under assumption
$n\ge 3$, see \cite{Hazrat_Zhang_multiple}. However, in that 
paper the proof was based on (a weaker version of) Theorem~A
and the usual (commutative!) localisation, so that there is no 
chance to make it work over arbitrary associative rings. Here, 
for quadruple commutators we assume 
that $n\ge 4$. The reason was that in \cite{NZ3} the proof of 
Lemma proceeds as follows. By Theorem~A and Lemma~3 one
only has to prove that 
$$ [y_{ij}(a,b),y_{hk}(c,d)]\in \big[E(n,A\circ B),E(n,C\circ D)\big], $$
\noindent
for $1\le i\neq j\le n$, $1\le h\neq k\le n$, $a\in A$, 
$b\in B$, $c\in C$, $d\in D$. By Lemma~13 conjugations by 
elements $x\in E(n,R)$ do not matter, since they amount to 
extra factors from the above triple commutators, which are
already accounted for. Now, for $n\ge 4$ this finishes the 
proof, since in this case modulo $E(n,R,C\circ D)$ 
we can move $y_{hk}(c,d)$ to a position, where it commutes 
with $y_{ij}(a,b)$, by Lemma~13 or Theorem~A. 

\begin{problem}
Prove that Lemma $18$ holds also for $n=3$, or construct a 
counter-example.
\end{problem}

Actually, it seems to us that either way it will be non-trivial. 
To prove the lemma one will have to verify that the commutator
$[y_{ij}(a,b),y_{ih}(c,d)]$ of two interlaced elementary 
commutators belongs where it should, and that's a non-trivial
calculation. On the other hand, since Lemma~\ref{l18} holds for 
quasi-finite rings, none of the usual counter-examples will
work, so that one will have to construct a truly non-commutative
counter-example. But to imitate Gerasimov's universal
counter-example would be an extremely troublesome business.

Observe that in fact the three subgroups lemma and 
Lemma~\ref{l17} imply the following poor man's version of our 
Theorem~1. In the commutative case it is essentially a
slight generalisation of a result by Himanee Apte and Alexei
Stepanov \cite{AS}, Lemma 3.4. 

\begin{proposition}\label{p1}
Let $R$ be an associative ring with identity $1$, $n\ge 3$, and 
let $A$,  $B$ and $C$ be its two-sided ideals. Then 
$$ [E(n,A\circ B),E(n,C)]\le [E(n,A\circ C),E(n,B)]\cdot
[E(n,A),E(n,B\circ C)]. $$
\end{proposition}

\begin{proof}
By the triple commutator formula of elementary subgroups
$$ [E(n,A\circ B),E(n,C)]=\big[[E(n,A),E(n,B)],E(n,C)\big]. $$
\noindent
By the three subgroups lemma 
\begin{multline*}
\big[[E(n,A),E(n,B)],E(n,C)\big]\le \\
\big[[E(n,A),E(n,C)],E(n,B)\big]\cdot \big[E(n,A),[E(n,B),E(n,C)]\big]. 
\end{multline*}
\noindent
Now, applying the triple commutator formula of elementary 
subgroups to the factors in the right hand side, we get
\begin{align*}
&\big[[E(n,A),E(n,C)],E(n,B)\big]=\big[E(n,A\circ C),E(n,B)\big],\\ 
&\big[E(n,A),[E(n,B),E(n,C)]\big]=\big[E(n,A),E(n,B\circ C)]\big]. 
\end{align*}
\end{proof}

By analogy, we can do the same applying the three subgroup
lemma to a quadruple commutator, and then combine it
with Lemma 15 again. Of course, this might be interesting
only for the case $n=3$.

\begin{proposition}\label{p2}
Let $R$ be an associative ring with $1$, $n\ge 3$, and let $A,B,C,D$ 
be two-sided ideals of $R$. Then
\begin{multline*}
\big[[E(n,A),E(n,B)], [E(n,C),E(n,D)] \big] \le\\
[E(n,(A\circ B)\circ C),E(n,D)]\cdot [E(n,(A\circ B)\circ D),E(n,C)]. 
\end{multline*}
\end{proposition}
\begin{proof}
Indeed, by the three subgroups lemma one has 
\begin{multline*}
\big[ [E(n,A),E(n,B)], [E(n,C),E(n,D)]\big] \le \\
\big[ \big[ [E(n,A),E(n,B)], E(n,C)\big],E(n,D)\big]\cdot
\big[ \big[ [E(n,A),E(n,B)], E(n,D)\big],E(n,D)\big].
\end{multline*}
\noindent
It remains to twice apply Lemma 15 to each factor.
\end{proof}

However, it is not feasible to prove Lemma 16 using calculations
at the level of subgroups. Rather, one should go to the level of 
individual elements. We have a blueprint how to do that by 
first applying the Hall---Witt identity, and then the identity in 
Theorem~1 to both resulting factors. However, the calculations
seem to be formidable, and as of today we have not succeeded.


\section{Inclusions among commutators for powers of one ideal}

Let us state the following exciting special case of Theorem 6.

\begin{proposition}\label{p3}
Let $I$ be an ideal of an associative ring $R$. Then
$$ [E(n,I^{r+s}),E(n,I^t)]\le 
[E(n,I^{r}),E(n,I^{s+t})]\cdot [E(n,I^{s}),E(n,I^{r+t})]. $$
\end{proposition}

In the case of an even exponent this gives an inclusion among 
two such commutators.

\begin{corollary}
Let $I$ be an ideal of an associative ring $R$. Then
$$ [E(n,I^{2r}),E(n,I^t)]\le 
[E(n,I^{r}),E(n,I^{r+t})]. $$
\end{corollary}

Iterated application of this proposition allows 
to establish all inclusions among the commutator subgroups
$[E(n,I^{r}),E(n,I^s)]$ of a given level. The following
remarkably easy argument was suggested to the authors
by Fedor Petrov. In the following theorem we call inclusions 
that result from Proposition~\ref{p3}, generic. It is not only 
interesting in itself, but also very relevant to obtain definitive 
results for Dedekind rings. These inclusions hold for
arbitrary associative rings. Of course, for a specific ring 
some of them may become equalities.

\begin{theorem}\label{t6}
Let $I$ be an ideal of an associative ring $R$, $m\ge 1$. 
Then the generic lattice of elementary commutator subgroups
$$ H(r)=[E(n,I^{r}),E(n,I^{m-r})]\le E(n,R,I^m),\qquad
0\le r\le m, $$
\noindent
of level $I^m$ is isomorphic to the lattice of divisors of $m$. 
In other words, generically,
$$ [E(n,I^{r}),E(n,I^{m-r})]\le [E(n,I^{s}),E(n,I^{m-s})]
\quad\Longleftrightarrow\quad \gcd(s,m) | \gcd(r,m). $$
\end{theorem}

\begin{proof} Let us understand $r$ in the definition of $H(r)$
modulo $m$. Then, clearly, one has $H(r)=H(m-r)=H(-r)$ and 
$H(r+s)\le H(r)H(s)$. Indeed, for $r,s\le m/2$ this is precisely
Proposition~\ref{p3}. When $r>m/2$ or/and $s>m/2$, we 
replace one 
or both of them by $m-r$ or/and $m-s$, and then apply 
Proposition~\ref{p3}.
\par
In particular, this means that $H(kr)\le H(r)$ for all $k\in\Int/m\Int$. 
Indeed, by induction $H(kr)\le H((k-1)r)H(r)\le H(r)$. This
establishes the second, and thus also the first claim of the
Theorem.
\end{proof}

\par
$\bullet$ In particular, this theorem implies that 
$$ H(r)=H(s)\quad\Longleftrightarrow\quad 
\gcd(r,m)=\gcd(s,m) $$
\noindent
and that
$$ H(r)\le H(s_1)\ldots H(s_l)\quad\Longleftrightarrow\quad 
\gcd(\gcd(s_1,m),\ldots,\gcd(s_l,m)) | \gcd(r,m). $$

\par\smallskip
$\bullet$ At level $I^p$, where $p$ is a prime, all non-trivial double commutators $ [E(n,I^r),E(n,I^s)]$, $r+s=p$, are equal.
\par\smallskip
$\bullet$ At level $I^4$ one has
$$ E(n,R,I^4)\le [E(n,I^2),E(n,I^2)]\le [E(n,I^3),E(n,I)]. $$
\noindent
The second claim of \cite{Mason_Stothers}, Theorem~5.4
asserts that the second inclusion may be strict! Actually, it is
strict already in the simplest example, where $R=\Int[i]$ is the
ring of Gaussian integers, and $I={\frak p}=(1+i)\Int[i]$ is the 
prime divisor of 2,
$$ [E(n,\Int[i],{\frak p}^2),E(n,\Int[i],{\frak p}^2)]< 
[E(n,\Int[i],{\frak p}^3),E(n,\Int[i],{\frak p})], $$
\noindent
of index 2. In other words, 
$$ [E(n,\Int[i],{\frak p}^2),E(n,\Int[i],{\frak p}^2)]=
E(n,\Int[i],{\frak p}^4), $$
\noindent
whereas 
$$ [E(n,\Int[i],{\frak p}^3),E(n,\Int[i],{\frak p})]=
\SL(n,\Int[i],{\frak p}^4). $$

\par\smallskip
$\bullet$ At level $I^6$ one has
$$ E(n,R,I^6)\le [E(n,I^3),E(n,I^3)], [E(n,I^4),E(n,I^2)]\le [E(n,I^5),E(n,I)], $$
\noindent
and there are no obvious inclusions between $[E(n,I^3),E(n,I^3)]$
and $ [E(n,I^4),E(n,I^2)]$. However, the third claim of \cite{Mason_Stothers}, Theorem~5.4
asserts in the above example of Gaussian integers one has
$$ [E(n,\Int[i],{\frak p}^4),E(n,\Int[i],{\frak p}^2)]=
E(n,\Int[i],{\frak p}^6), $$
\noindent
whereas 
$$ [E(n,\Int[i],{\frak p}^3),E(n,\Int[i],{\frak p}^3)]=
[E(n,\Int[i],{\frak p}^5),E(n,\Int[i],{\frak p})], $$
\noindent
is strictly larger, being a proper intermediate subgroup between
$E(n,\Int[i],{\frak p}^6)$ and $\SL(n,\Int[i],{\frak p}^6)$, both
indices being equal to 2.

\par\smallskip
$\bullet$ At level $I^{30}$, any of the three subgroups
$$   [E(n,I^6),E(n,I^{10})],\quad  [E(n,I^6),E(n,I^{15})],\quad
[E(n,I^{10}),E(n,I^{15})] $$
\noindent
is contained in the product of the other two.


\section{When $[E(n,A+B),E(n,A\cap B)]=[E(n,A),E(n,B)]$?}

For the ideals themselves, one has an obvious inclusion 
$$ (A+B)\circ(A\cap B)=(A+B)(A\cap B)+(A\cap B)(A+B)\le 
AB+BA=A\circ B. $$
\noindent
Only very rarely this inclusion is always an equality. In fact, it is
classically known that among
commutative integral domains $(A+B)(A\cap B)=AB$ characterises 
Prüfer domains.
\par
A Noetherian Pr\"ufer domain is a Dedekind domain, so any
Noetherian domain that is not Dedekind provides a counterexample
to the equality.
Let, for instance, $R=K[x,y]$, $A=xR$, $B=yR$. Then $A+B=xR+yR$, whereas
$A\cap B=AB=xyR$, since $R$ is factorial and $x$ and $y$ are
coprime. Thus, 
$$ (A+B)(A\cap B)=x^2yR+xy^2R<AB. $$

The following inclusion can be verified in the style of the original 
proof of Lemma 7, see \cite{Vavilov_Stepanov_revisited}. But 
it is also an immediate corollary of Lemmas~2--6.

\begin{proposition}\label{p4}
For any two ideals $A,B\unlhd R$, $n\ge 3$, one has
$$ [E(n,A+B),E(n,A\cap B)]\le [E(n,A),E(n,B)] $$
\end{proposition}

\begin{proof}
Lemma 3 and the formula at the beginning of this section show that
the level of the left hand side is contained in the level of the right
hand side,
$$ E\big(n,R,(A+B)\circ(A\cap B)\big)\le E(n,R,A\circ B). $$
\par
Thus, it only remains to prove that the elementary commutators
$y_{ij}(a+b,c)$, where $a\in A$, $b\in B$, $c\in A\cap B$,
in the left hand side belong to the right hand side.
\par
By Lemma 5, one has
$$ y_{ij}(a+b,c)\equiv 
 y_{ij}(a,c)\cdot y_{ij}(b,c)
\pamod{E\big(n,R,(A+B)\circ(A\cap B)\big)}. $$
\noindent
Thus, this congruence holds also modulo the larger subgroup
$E(n,R,A\circ B)$. 
\par
On the other hand, Lemma 4
implies that 
$$ y_{ij}(b,c)\equiv y_{ij}(c,-b)
\pamod{E(n,R,A\circ B)}. $$ 
\par
Combining the above congruences, we see that
$$ y_{ij}(a+b,c)\equiv  y_{ij}(a,c)\cdot y_{ij}(c,-b)\pamod{E(n,R,A\circ B)}, $$ 
\noindent
where both commutators in the right hand side belong to
$[E(n,A),E(n,B)]$, which proves the desired inclusion.
\end{proof}

Of course, when $A$ and $B$ are comaximal, by Lemma 7 one has
$$ [E(n,A+B),E(n,A\cap B)]=[E(n,A),E(n,B)]. $$
\noindent
Indeed, in this case $A\cap B=AB$ so that both sides are
equal to $E(n,R,AB)$.  This is also true in the opposite 
case, when $A=B$, as in all counter-examples listed in
\cite{NZ3}. But, as we've seen, in general it may break
already as regards the levels of these subgroups, since
the level of the left hand side may be strictly smaller, than
the level of the right hand side. Thus, the question remains

\begin{problem}
When
$$ [E(n,A+B),E(n,A\cap B)]=[E(n,A),E(n,B)]? $$
\end{problem}

In a subsequent paper we will show that this equality, and 
in fact much more general statements, hold for Dedekind 
rings.


\section{Intersections of elementary subgroups}

Let $A$ and $B$ be two ideals of a commutative ring $R$,  $n\ge 3$.
Clearly,
$$ [E(n,A),E(n,B)]\le E(n,R,A)\cap E(n,R,B)]. $$
\noindent
There is no obvious counter-example to the following stronger claim.

\begin{problem}
Is it true that for all ideals $A,B\unlhd R$, $n\ge 3$, one has
$$ [E(n,A),E(n,B)]\le E(n,R,A\cap B)]. $$
\end{problem}

This is obviously true when $[E(n,A),E(n,B)]=E(n,R,AB)$.
But in all examples where $[E(n,A),E(n,B)]>E(n,R,AB)$ we are
aware of, one still has 
$$ [E(n,A),E(n,B)]\le E(n,R,A\cap B). $$ 
\noindent
In these examples usually $A=B$, when the above inclusion is
obvious.
\par
Dually to the Corollary 3 of Lemma 3 one has
$$ \GL(n,R,A)\cap\GL(n,R,B)=\GL(n,R,A\cap B), $$
\noindent
this equality is classically known, and obvious. The same
holds also for congruence subgroups in $\SL(n,R)$.

However, a similar statement for sums of ideals is 
obviously false for $\GL(n,R,A)$, already in the case
$R=\Int$. In other words, in general $\GL(n,R,A)\GL(n,R,B)$
is strictly smaller than $\GL(n,R,A+B)$, even for comaximal
$A$ and $B$. The trivial reason is that $R$ may have more
invertible elements than just those that can be expressed 
as products of invertible elements congruent  to 1 modulo 
$A$ or modulo $B$.
\par
In fact, even for the easier case of $\SL(n,R)$ the equality
$$ \SL(n,R,A)\SL(n,R,B)=\SL(n,R,A+B) $$
\noindent
only holds under some very strong assumptions, such as one
of the factor-rings $R/A$ or $R/B$ being semi-local\footnote{This
is automatically the case, for instance, for non-zero ideals
in Noetherian integral domains of dimension 1. Say, for
Dedekind rings.}, see \cite{Bass}, Corollary 9.3, p.~267,
or \cite{Mason_Stothers}, Theorem~2.2. 

Also the corresponding property for intersections of elementary
subgroups fails in general.

\begin{proposition}\label{p5}
For two ideals $A,B\unlhd R$ the group $E(n,R,A\cap B)$ can be 
strictly smaller than $E(n,R,A)\cap E(n,R,B)$. 
\end{proposition}

\begin{proof}
Here is the smallest such example. Let $R$ be the ring of
integers of the imaginary quadratic field $\Rat\left(\sqrt{-7}\right)$.
Then $R=\Int[\zeta]$, where $\zeta=\displaystyle{1+i\sqrt{7}\over 2}$
and $R^*=\mu(R)=\{\pm 1\}$.
\par
Now, set ${\frak p}_1=\zeta R$, and ${\frak p}_2=\overline{\frak p}_1=
\overline{\zeta}R$.  Since $\zeta+\overline{\zeta}=1$, the
ideals ${\frak p}_1$ and ${\frak p}_2$ are coprime (= comaximal, 
in this case). And since $\zeta\cdot\overline{\zeta}=2$, one has
$2={\frak p}_1{\frak p}_2$ so that the prime 2 completely decomposes in $R$.
\par
Recall the formula of Bass--Milnor--Serre, \cite{Bass_Milnor_Serre}
for the exponent of the $p$-part of the order of $\SK_1(R,I)$:
$$ v_p\big(|\SK_1(R,I)|\big)=\min_{{\frak p}|p}\left[
{v_{{\frak p}}(I)\over v_{{\frak p}}(p)}- {1\over p-1}\right]_{[0,v_p(\mu(R))]}. $$
\noindent
Here the minimum is taken over all prime divisors of $p$ in $R$,
while $[x]_{[0,m]}$ is the closest integer in the interval $[0,m]$
to the integer part $[x]$ of $x$.
\par
Now, set $A={\frak p}_1^2$, $B={\frak p}_2^2$. The ideals $A$ 
and $B$ are still comaximal, $A+B=R$. In particular, $A\cap B=AB$.
\par
Now, this formula implies that $\SK_1(R,A)=\SK_1(R,B)=1$,
in other words,
$$ E(n,R,A)=\SL(n,R,A),\qquad E(n,R,B)=\SL(n,R,B), $$
\noindent
and thus
$$ E(n,R,A)\cap E(n,R,B)=
\SL(n,R,A)\cap\SL(n,R,B)=\SL(n,R,AB). $$
\par
On the other hand, $\SK_1(R,A\cap B)=\SK_1(R,AB)=\{\pm 1\}$,
so that the subgroup $E(n,R,A\cap B)=E(n,R,AB)$ has index 2 in 
$E(n,R,A)\cap E(n,R,B)$.
\end{proof}

Of course, since $A$ and $B$ are comaximal, by Lemma 7 one has 
$$ [E(n,A),E(n,B)]=E(n,R,AB)=E(n,R,A\cap B), $$ 
\noindent
so that we do not get a counter-example to Problem 3. By the
same token, there are no such counter-examples for imaginary 
quadratic rings. On the other hand, Lemma 9 the last equality 
always holds for Dedekind rings of arithmetic type with {\it infinite\/} 
multiplicative group, so that in this case there are no 
counter-examples to Problem 3 either.


\section{Final remarks}

It would be natural to generalise results of the present paper
to more general contexts.

\begin{problem}
Generalise Theorem~$1$ and other results of the 
present paper to Chevalley groups.
\end{problem}

We do not see any difficulties in treating the simply laced 
case. However, for doubly laced systems and for type $\G_2$
one might get longer and fancier formulas, than those
in Theorem~1. 

\begin{problem}
Generalise Theorems~$1$ and other results of the present 
paper to Bak's unitary groups.
\end{problem}

It was a great experience to collaborate in this field with 
Roozbeh Hazrat and Alexei Stepanov over the last decades. 
Also, we are very grateful to Pavel Kolesnikov for his questions 
during our talk, and to Fedor Petrov for suggesting the above 
proof of Theorem 6.


\end{document}